\documentclass[a4paper,12pt]{amsart}
\usepackage{amsmath}
\usepackage{array}
\usepackage{amsthm}
\usepackage{amssymb}
\usepackage{graphicx}
\usepackage{enumerate}
\usepackage{color}
\usepackage{multicol}
\usepackage[colorlinks=1, bookmarks=true]{hyperref}

\usepackage[all]{xy}

\theoremstyle{plain}
\newtheorem*{thm*}{Theorem}
\newtheorem*{prop*}{Proposition}
\newtheorem*{rem*}{Remark}

\newtheorem{thm}{Theorem}[section]
\newtheorem{cor}[thm]{Corollary}

\newtheorem{prop}[thm]{Proposition}
\newtheorem{lm}[thm]{Lemma}

\newtheorem{obs*}{Observation}

\newtheorem{claim*}{Claim}

\numberwithin{equation}{thm}

\theoremstyle{remark}
\newtheorem{rem}[thm]{Remark}

\newcommand{\kk}{\Bbbk} 

\newcommand{\CC}{{\mathcal{C}}}
\newcommand{\D}{{\mathcal{D}}}

 \def\HIm{\ensuremath{ { \rm HIm}}}
  
  \def\Im{\ensuremath{{ \rm im}}}

\let\phi\varphi
\def\D{\ensuremath{\Delta}}

\def\ot{\ensuremath {{\otimes}}}

\let\epsilon\varepsilon

\makeatletter
\@addtoreset{equation}{section}
\makeatother

 \newcommand\Heq{\mathop{\hbox{\rm Heq}}\nolimits}

  \newcommand\Hker{\mathop{\hbox{\rm Hker}}\nolimits}

  \newcommand\Hcoeq{\mathop{\hbox{\rm Hcoeq}}\nolimits}

  \newcommand\Hcoker{\mathop{\hbox{\rm Hcoker}}\nolimits}

  \newcommand\heq{\mathop{\hbox{\rm heq}}\nolimits}

  \newcommand\hker{\mathop{\hbox{\rm hker}}\nolimits}

  \newcommand\hcoeq{\mathop{\hbox{\rm hcoeq}}\nolimits}

  \newcommand\hcoker{\mathop{\hbox{\rm hcoker}}\nolimits}

  \newcommand\id{\mathop{\hbox{\rm id}}\nolimits}

\let\D\Delta
\let\e\epsilon

\def\Pt{{\mathcal Pt}}
\def\Hfc{{\mathcal H^{\rm coco}}}
\def\Hfcc{{\mathcal H^{\rm co-coco}}}
\newcommand{\invamalg}{\mathbin{\text{\,\raise-0.3ex\hbox{\rotatebox[origin=c]{180}{$\amalg$}}}}}

 \let\act\rightharpoonup

\begin{document}

 \title{On some properties of the category of  cocommutative Hopf Algebras}

\author{Christine Vespa}\thanks{vespa@math.unistra.fr}
 \author{ Marc Wambst}\thanks{wambst@math.unistra.fr}
\address{ 
Institut de Recherche Math\'ematique Avanc\'ee\\ 
UMR 7501 de l'universit\'e de Strasbourg et du CNRS\\  
 7 rue Ren\'e-Descartes\\
  67084 Strasbourg Cedex, France 
}


\date{\today}

\begin{abstract}
By a recent work of Gran-Kadjo-Vercruysse, the category of cocommutative Hopf algebras over a field of characteristic zero is semi-abelian. In this paper, we explore some properties of this categoy, in particular we show that its abelian core  is the category of commutative and cocommutative Hopf algebras. 
 \medskip

 \textit{Mathematics Subject Classification:} 
 16T05,  18B99, 18E10.
\smallskip

\textit{Keywords}: Hopf algebra, abelian category, semi-abelian category
\end{abstract}
\maketitle

\section*{Introduction}
It is a classical result that the category of commutative and cocommutative Hopf algebras is an abelian category (see for example \cite[Corollary 4.16]{Take} or \cite[Theorem 4.3]{N}). It is also   known that this is no more the case for the category of cocommutative (resp. commutative) Hopf algebras since the coproduct and the product are not equivalent in each of these categories.
  
In 2002, the more general notion of semi-abelian category emerges in category theory \cite{Jane}.
In a semi-abelian category, classical diagram lemmas (five lemma, snake lemma…) are valid. Among the examples of semi-abelian category we have the categories of groups, ring without unit, Lie algebras (and more generally algebras over a reduced linear operad) and sheaves or presheaves of these. Abelian categories are also examples of semi-abelian categories. In fact, a category $\CC$ is abelian precisely when both $\CC$ and $\CC^{op}$ are semi-abelian. Since then, semi-abelian categories become widely-known as the good generalization of the category of groups just as abelian categories is the good generalization of the category of abelian groups. 
 
 A  category is semi-abelian if it has a zero object and finite products and is Barr-exact \cite{Barr} and protomodular in the sense of Bourn \cite{Bourn}. For more details on exact, protomodular and semi-abelian categories, we refer the reader to the excellent book by Borceux and Bourn \cite{BB}.

In \cite{GKV} Gran, Kadjo and Vercruysse prove the following theorem.
\begin{thm} \label{thm}
The category of cocommutative Hopf algebras over a field of  characteristic zero is semi-abelian. 
\end{thm}

 In this paper we compute the abelian core of this semi-abelian category ({\it i.e.} the subcategory of abelian objects) and obtain the following result:
 
 \begin{thm} \label{thm2}
 The abelian core of the category of cocommutative Hopf algebras over a field of characteristic zero is the category of commutative and cocommutative Hopf algebras.
 \end{thm} 
 
   Roughly speaking the abelian core of a semi-abelian category $\CC$ is the biggest abelian subcategory of $\CC$. In particular, we recover as a corollary of this theorem the fact that the category of commutative and cocommutative Hopf algebras is abelian. It also proves that the category of cocommutative Hopf algebras is not abelian.

 In this paper,  we will follow the   
  characterization of semi-abelian categories  
   given by Hartl and Loiseau in  \cite{HL}. Namely, a  category $\mathcal C$ is semi-abelian if and only if the following  four axioms are satisfied.
 
 \begin{enumerate}[{A}1.]
 \item    The category $\mathcal C$ is pointed, finitely complete and  finitely cocomplete.
  \item    For any split epimorphism $p : X \rightarrow Y$ with section $s : Y \rightarrow X$ and with kernel $\kappa: K\hookrightarrow X $, the arrow $<\kappa, s>: K\amalg Y\rightarrow X$ is a cokernel.
  \item The pullback of a cokernel  is a cokernel.
  \item The image of a kernel by a cokernel is a kernel.
 \end{enumerate}

  In the first sections of this  work  we will prove axioms $A1$, $A2$ and $A3$.  The proofs of axioms $A2$ and $A3$ are heavily based on a result of Newman \cite{N} (see also \cite{Masuoka}) giving a natural correspondence between left ideal which are also two-sided coideals of a cocommutative Hopf algebra and its sub-Hopf algebras.
Axiom $A4$ corresponding to \cite[Theorem 3.7]{GKV}, one recovers Theorem \ref{thm}.
The last section is devoted to the proof of Theorem~\ref{thm2}.

 The category of commutative Hopf algebras is not semi-abelian. However, we will prove in \cite{BVW} that the opposite of this category is semi-abelian.
  
  \textbf{Acknowledgement} : The authors are very grateful to Dominique Bourn for his interest in this work, for very fruitful discussions and for his many helpful comments. They also thank Dominique Bourn and Marino Gran for both pointing out mistakes in previous versions of this paper. More subtile properties of categories of Hopf algebras will be study in the forthcoming paper \cite{BVW}, which will surpass this one.

\section{Conventions and prerequisites.}

In the whole article,  $\kk$ is a commutative field.  By {\it module} we will understand   module over $\kk$.  
 The unadorned symbol $\ot$ between two $\kk$-modules will   stand  for $\ot _\kk$.     
We denote  by $\Hfc$ the category of cocommutative Hopf algebras over $\kk$ and by
  $\Hfcc$ the category of commutative and cocommutative Hopf algebras over $\kk$.

Let $H$ be a  Hopf algebra. Its structure maps will be denoted as follows:  multiplication $\mu_H:H\ot H \to H$, comultiplication $\Delta_H:H\to H \ot H$, unit $\eta_H: \kk \to H$,    counit $\e_H :H\to \kk$ and antipode $S_H:H\to H$. 
Moreover, for any $a,b\in H$, we will denote $ \mu_H(a\ot b)$ by $ab$. The unity $\eta_H(1)$ will be denoted by $1_A$ or simply $1$. We also adopt the Sweedler-Heyneman  notation   $\D_H(a)= a_1\ot a_2$.  More generally,  a generic element in a tensor product of $\kk$-modules $A\ot B$, will be denoted by $a\ot b$, the summation sign being omitted.

  We will call {\it Hopf ideal} of a Hopf algebra $H$ any two-sided ideal $I$ of the algebra $H$   which is also a two-sided coideal of the coalgebra $H$   ({\it i.e. }$ 
 \D_H(I)\subset I\ot H+H\ot I
$ and $\e_H(I)= 0$)  such, moreover, one has $S_H(I)\subset I$. In particular,  the   structure on $H$ induces a Hopf algebra structure on the quotient $H/I$.

 A   sub-Hopf algebra  $A$ of a Hopf algebra $H$ will be called {\it  normal  } if, for any $a\in A$ and $y\in H$, one has $y_1aS(y_2)\in A$. In particular, when $H$ is commutative, one has $y_1aS(y_2)=y_1 S(y_2)a=\e(y)a$. Thus,  all    sub-algebras of $H$ are normal.  

We  will denote by $\Im(\phi)$ the linear image and $\ker(\phi)$ the linear kernel of any morphism of Hopf algebras $\phi$. A morphism of Hopf algebras $\phi$ is injective if $\ker(\phi)=0$. The kernel $\ker(\e_H)$ of the counit of a Hopf algebra $H$ will more specifically be denoted by $H^+$.

  \begin{section}{Completeness and cocompleteness}
  
 In this section we prove Axiom A1.
 
 \begin{thm}
 The category $\Hfc$ is pointed, finitely complete and  finitely cocomplete.
\end{thm}

\begin{proof}
  First, we remark  that the category  $\Hfc$ is pointed. Indeed, its zero object is the ground field $\kk$  with initial and terminal morphisms given by the unit $\eta_H:H\to \kk$ and the counit $\e_H:\kk\to H$.
  
 Finite (co)completeness is the existence of finite (co)limits which is equivalent to the existence of finite (co)products and (co)equalizer.  For the finite (co)products, as the category is pointed, it is in fact  sufficient to prove the existence of binary (co)products. Details can be found in  \cite[V.2]{ML}. 
 
The explicit descriptions of the   binary (co)products and (co)\-equalizer in $\Hfc$ are given below. The reader may check, by straightforward computations, that the given constructions fulfill the definitions.   
\end{proof}

 The definition of equalizers for morphisms of general Hopf algebras was first given by Andruskiewitsch and Devoto  in  {\cite{AD}}  generalizing the notions of kernel sooner  given in     \cite{Sw} or \cite{BCM}. The same authors give  explicit description of coequalizers and cokernels.
 For finite coproducts of Hopf algebras we refer to \cite[\S 2]{Pb}   and for products to \cite{A}, \cite{BW}, \cite{Acons}. We simply follow   the cited authors.  It happens that their   constructions for Hopf algebras restricts to  $\Hfc$.

 \begin{subsection}{Equalizers, kernels  and products} \
  
  First,  we  give the constructions of equalizers and kernels in $\Hfc$.  By \cite[Lemma~1.1.3]{AD},
 for any   two morphisms $ \xymatrix{  A   \ar@<0.5ex>[r]^{f}\ar@<-0.5ex>[r]_{g}& B}$  of $\Hfc$  the    set 
 $$\Heq(f,g)=
\{x\in A \mid  f(x_1)\ot x_2=g(x_1)\ot x_2 \}=     \{x\in A \mid  x_1\ot f(x_2)=x_1\ot g(x_2)  \}  
 $$ 
 is a   sub-Hopf algebra of $A$. 
The equalizer of  $ \xymatrix{  A   \ar@<0.5ex>[r]^{f}\ar@<-0.5ex>[r]_{g}& B}$ in $\Hfc$ is the inclusion morphism
$ \heq(f,g): \Heq(f,g)\to A$.

We denote by $\hker(f)$ the kernel of 
a morphism $ \displaystyle A \xrightarrow[]{f} B$  in $ \Hfc$ which is, by definition, the equalizer  
  $ \heq( f,   \eta_B{\circ} \e_A)$. Explicitly, $\hker(f)$ is the inclusion $ \Hker(f)\to A$   with 
  $$\Hker(f)= \{x\in A \mid  x_1\ot f(x_2)=x\ot 1)   = \{x\in A \mid  f(x_1)\ot x_2=1\ot x  \}.$$

 It can be easily check by straightforward computation that 
 the kernel $\Hker(f)$  of 
a morphism $ \displaystyle A \xrightarrow[]{f} B$ is a normal sub-Hopf algebra of $A$.  

The direct product of two objects $A$ and $B$ in $\Hfc$  is given by the tensor product over~$\kk$. Indeed the product is $<\pi_A,\pi_B>$ where the projections 
  
  $$ \xymatrix{   A &   A\ot B    \ar@<0ex>[r]^{\ \ \pi_B}  \ar@<0ex>[l]_{\pi_A} 
 & B }     $$
  are given by $\pi_A(a\ot b)= \e_B(b)a$ and $\pi_A(a\ot b)= \e_A(a)b$ for $a\ot b\in A\ot B$.
   For any two morphisms $f:H\to A$ and $g:H\to B$ of $\Hfc$, the morphism $\varphi:H\to A\ot B$ fulfilling the universal property    of the product is defined by $\varphi(x)= f(x_1)\ot g(x_2)$.

  This form of the product is very specific of the cocommutative case. It is a consequence of the fact that the comultiplication $\Delta_H$ of a Hopf algebra $H$ is   a morphism of coalgebras if and only if  $H$ is cocommutative.  
   
    \end{subsection}

\begin{subsection}{Coequalizers, cokernels  and coproducts}
   
We  now  give   explicit description of coequalizers and cokernels in $\Hfc$. 
   Let $A$ and $B$  be two cocommutative Hopf algebras. For any   two morphisms $ \xymatrix{  A   \ar@<0.5ex>[r]^{f}\ar@<-0.5ex>[r]_{g}& B}$, set 
  $$
  J=\{f(x)-g(x)\mid x\in A\} 
  \text { \quad and \quad} 
  \Hcoeq(f,g)=B/BJB  
  $$
The projection  $\hcoeq(f,g): B\to \Hcoeq(f,g)$ is the 
     coequalizer  of     $ \xymatrix{  A   \ar@<0.5ex>[r]^{f}\ar@<-0.5ex>[r]_{g}& B}$, set 
  $$
  J=\{f(x)-g(x)\mid x\in A\} 
  \text { \quad and \quad} 
  \Hcoeq(f,g)=B/BJB  
  $$
  in $\Hfc$.   
     In the sequel, for simplicity of notations, we sometimes denote the Hopf ideal $BJB$ by $\langle J\rangle$.

  The cokernel  $\hcoker(f)$ of 
a morphism $ \displaystyle A \xrightarrow[]{f} B$  in $ \Hfc$   is, by definition, the coequalizer  
  $ \hcoeq( f,   \eta_B{\circ} \e_A)$. 
    Notice that, when $g= \eta_B{{\circ}}\epsilon_A$,  the set $J$ reduces to $J= f(A^+)$ where $A^+=\ker(\epsilon_A)$ is the linear kernel of the counit of $A$. So the cokernel of $f$ is the projection map
    $$
    \hcoker(f): B \to B/\langle f(A^+)\rangle .
    $$
 We set $\Hcoker(f)= B/\langle f(A^+)\rangle $. 
 
 \bigskip
 Let $A$ and $B$ be two objects in  $\Hfc$. The coproduct object $A\amalg B$ of $A$ and $B$ in  $\Hfc$ has the following explicit description 
 (\cite{Pb} or \cite{Acons}). The coproduct $A\amalg B$ is the  module  spanned over $\kk$ as an algebra by the elements $1$,  $t_{a}$ and $t_{b}$ with $a\in A $ and $b\in B$ submitted to the relations 
 $$t_{\lambda}=\lambda,\quad
  t_{ \lambda a + b}= \lambda t_{   a} + t_{ b}
  ,\quad
  t_{aa'}=  t_{a}t_{a'}, \quad t_{bb'}=  t_{b}t_{b'}
  $$
with $\lambda\in \kk$, with $a,a'\in A$, and $b,b'\in B$.

   The coproduct  $A\amalg B$   is endowed with a Hopf algebra structure given by:
 
 $$
 \Delta(t_{a})=  \sum t_{a_1}\ot t_{a_2}   ,\quad
 \epsilon(t_{a})=\epsilon(a)  ,\quad
 S(t_{a})= t_{S(a)}.
 $$
 with $a\in A$ or $a\in B$. The coproduct diagram is  given by 
  $$ \xymatrix{    A \ar@<0ex>[r]^{\iota_A\ \ \ } & A\amalg B &
B\ar@<0ex>[l]_{\quad \iota_B } } 
 $$
 with $\iota_A(a)= t_{a}$ and $\iota_B(b)= t_{b}$ for $a\in A$, and $b\in B$.
This construction  satisfies the   universal  property of coproduct. Indeed, for any  two morphisms $f: A\to H$ and $g: B\to H$ of $\Hfc$, the unique morphism $ h: A\amalg B\to H$  such as one has $h{\circ}\iota_A=f$ and  $h{\circ}\iota_B=g$ 
is defined by $h(t_{a})= f(a)$ and $h(t_{b})= g(b)$ on the generators of $A\amalg B$ with $a\in A$ and $b\in B$.

 \end{subsection}
 
 \end{section}

 \begin{section} {The Newman correspondence,  the semi-direct product}
 
In this section,  we  recall some constructions and results involving kernel and cokernels which we will use in the sequel.

 \begin{subsection}{The Newman correspondence} 
  
The following result is crucial in the sequel of the paper

\begin{thm}\cite{N}
For any   cocommutative   Hopf algebra over a field, there is a one-to-one correspondence between its sub-Hopf algebras  and its left ideals  which are also two-sided coideals. 
\end{thm}
  
 If $G$ is a sub-Hopf algebra of a Hopf algebra $H$, Newman associates the ideal  $\tau(G)= HG^+$.
 He proves that $\tau$ is bijective with reciprocal $\sigma(I)= \Hker(  H\to H/I)$.

 We state three lemmas  directly induced by this result. 
 
 \begin{lm} \label{lm1}
 Let $H$ be a cocommutative Hopf algebra over  a field. For any Hopf ideal $I$ of $H$, it exists a sub-coalgebra $G$ of $H$ such that one has the isomorphism of cocommutative Hopf algebras
 $$
 H/ I\cong H/HG^+H.
  $$
 \end{lm}
 In other words the projection  map $H \to  H/I$ is a cokernel in the category $\Hfc$. 
 
\begin{proof}  The two-sided ideal $I$ is in particular a  left ideal. So after   \cite[Corollary 3.4]{N}, one has  $I \cong HG^+$ for $G= \Hker(   H\to H/I)$. As $I$ is also a right ideal, we deduce $I=  IH\cong   HG^+H$.  
 \end{proof}
   
   As a consequence, we have: 
 
 \begin{lm} \label{lm1-bis}
 Let $f:H\to H'$ be a surjective map of cocommutative Hopf algebras over  a field.  
 The map $f$ is a cokernel in the category $\Hfc$.
  \end{lm}

\begin{proof}
Consider the linear ideal $I=\ker(f)$ of $f:H\to H'$.  It is well know that $I$ is a Hopf ideal. Moreover, one has $H/I\cong H'$.
After Lemma \ref{lm1}, the map is a cokernel.  \end{proof}

  \begin{lm} \label{lm2}
  Let $H$ be a cocommutative Hopf algebra over a field and $G$ one of its normal sub-Hopf algebras.  Then  we have $G\cong \Hker( H\to H/HG^+H)$. 
  \end{lm}
  In other words any inclusion of a normal sub-Hopf algebra into a Hopf algebra is a kernel. 

 \begin{proof}  As $G$ is a normal sub-Hopf algebra, the usual trick 
  $ 
 gh= h_1 (S(h_2)gh_3)
 $ for $g\in G$ and $h\in H$  shows that $HG^+H= HG^+$. Moreover one has $HG^+= H\!\left( \Hker( H\to H/HG^+H)\right)^+$ by    \cite[Corollary 3.4]{N}     and,  as $\tau$ is injective by \cite[Corollary 2.5]{N}, on has 
 $G\cong   \Hker( H\to H/HG^+H)$.
 \end{proof}

 We also point out an important lemma which can be found as a part of the proof of \cite[Theorem 4.4]{N}
 
  \begin{lm} \label{lm3}
  A monomorphism in $\Hfc$ is injective. 
    \end{lm}
 
 \begin{proof}
 Let $m: X\to Y$ be a monomorphism in $\Hfc$. The linear kernel $\ker(m)$ of $m$ is a Hopf ideal of $X$ by \cite[Theorem 4.17]{Sw}. By Lemma \ref{lm1}, it exists a Hopf algebra inclusion $ \xymatrix{  G   \ar [r]^{\iota} & X}$ such as $X\iota(G^+)X=\ker(m)$. This implies $(m{\circ} \iota)(g)= \e_X(g)1_Y= (m{\circ} \eta_Y{\circ}\e_X{\circ }\iota)(g) $ for $g\in G$. So we get $  \iota =\eta_Y{\circ}\e_X{\circ }\iota$ and thus we have  $G=\kk$. Finally, one has $\ker(m)= X\kk^+X=\{0\}$.
 \end{proof}

 \end{subsection}

 \subsection{The semi-direct product of Hopf algebras}

 We will need a notion of semi-direct product of Hopf algebras. The construction we will recall for our purpose is  very special case of the now classical semi-direct product introduced in the article \cite{BCM} to which we refer for proofs. Anyway properties of the product we state here    may also be     checked through direct calculation. 
  
 Let $Y$ and $K$ be two Hopf algebras, we will say that $K$ is a {\it $Y$-Hopf algebra} if   $Y$ acts on $K$. In other words if it exists an action map $\hbox{--}\act\hbox{--}: Y\ot K \to K  $ such as the following axioms are satisfied: 
 
\begin{eqnarray*}
y\act (ab)=( y_1\act a)(y_2\act b) &\qquad&  1_Y\act a=a\\
(yy')\act a= y\act(y'\act a)&& y\act 1_K= \e_Y(y)1_K\\
\end{eqnarray*}
 with $y,y'\in Y$ and  $a,b\in K$.   We denote by 
 $Y\hbox{-}\Hfc$ the  subcategory of $\Hfc$ having as objects the  $Y$-Hopf algebras and as morphisms those of $\Hfc$ compatible with the action. 
 
 Given  an object $K$ in  $Y\hbox{-}\Hfc$,   one may define a semi-direct product $K\#Y$ of $K$ and $Y$. As a module it is $K\ot Y$ on which a structure of Hopf algebras is endowed by 
 
  \begin{eqnarray*}
 (a\ot y)(b\ot y')&=& a (y_1 \act b) \ot y_2y'   \\
 \eta(1)&=& 1\ot 1\\
  \D( a\ot y)&=& ( a_1\ot y_1)\ot ( a_2\ot y_2) \\
   \e(a\ot y)&=& \e(a)\e(y)\\
 S(  x\ot y )&= & (S(y_1)\act S(a)) \ot S(y_2)  
 \end{eqnarray*}
 
given for $a,b\in K$ and $y,y'\in Y$.
 This product is nothing else than the product $ K\#_\sigma Y$ of \cite{BCM} with cocycle $ \sigma= \eta_Y{\circ}(\e_K\ot\e_K) :   K\ot K \to Y $.  In the sequel, an element $a\ot y \in  K\#Y$ will be denoted by $a\#y$.

Consider the category $\Pt_Y$ of pointed objects over an object $Y$ of $\Hfc$. 
 Its objects are the couples of maps  $ (p,s): \xymatrix{  X\ar[r]_{ p } & Y \ar@/_0.7pc/[l]_{s}} $ of $\Hfc$ such that $s$ is a section of $p$ ({\it i.e.} $p{\circ} s= \id_Y)$. The morphisms of $\Pt_Y$ between two objects $ \xymatrix{  X\ar[r]_{ p } & Y \ar@/_0.7pc/[l]_{s}} $ and  $\xymatrix{  X'\ar[r]_{ p' } & Y \ar@/_0.7pc/[l]_{s'}} $   are the maps $f:X\to X'$ such has one has $p'{\circ} f=p$ and $f{\circ} s=s'$.

 \begin{lm} \label{lm4}
 Let $Y$ be an object of $\Hfc$. The categories $\Pt_Y$ and $Y\hbox{-}\Hfc$ are equivalent. \end{lm}

\begin{proof} Details may be found in \cite{BCM}, we will only  describe the correspondence between objects.
To an action   $\hbox{--}\act\hbox{--}: Y\ot K \to K  $ one associates the maps 
$$
p= \e_K\ot \id_Y: K\# Y\to Y \text{\quad and\quad } s= \eta_K\ot \id_Y: Y\to K\#Y.
$$

 On the other hand, given the data $\xymatrix{  X\ar[r]_{ p } & Y \ar@/_0.7pc/[l]_{s}} $, one sets $K=\Hker(p)$. It is easy to check that $y\act k= s(y_1)ks(S_Y(y_2))$ defines an action of $Y$ on $K$.
 
 The equivalence is based on the isomorphism  between   $K{ \#} Y$  and  $X$     given by 
$$
\begin{aligned}
\begin{minipage}{150pt}  $\begin{aligned}\mathcal F: \ &X\to  K\# Y  \\  &x \mapsto   x_1 s(Sp(x_2))\# p(x_3)\end{aligned} 
$
\end{minipage} && 
\begin{minipage}{150pt}  $ \begin{aligned}\mathcal G: \ &   Y\# K \to X \\  &  k\# y \mapsto  ky\end {aligned} $ 
\end{minipage} 
\end{aligned}
$$
\end{proof}

 \end{section}

 \begin{section}{Pullbacks of cokernels, Regularity}

This section is devoted to the proof of Axiom A3.  As a consequence,   one deduces  that $\Hfc$ is a regular category and a homological category.

From the constructions of products and equalizers one easily derives the  one for 
   pullbacks in $\Hfc$.  
Let $A$, $B$, $C$ be cocommutative  Hopf algebras and 
   $f: A\to C$ and $g:B\to C$  be morphisms of Hopf algebras. 
The pullback object of  $A$ and $B$ over $C$ is given by 
$$ A{\invamalg}_C B= \{ a\ot b\in A\ot B \mid a_1\ot f(a_2)\ot b= a\ot g(b_1)\ot b_2 \}.$$
It is a   sub-Hopf algebra of the product Hopf algebra   $A\ot B$. 
One has the commutative diagram 
$$
  \xymatrix{
      A{\invamalg}_CB \ar[r]^{\ \pi_B} \ar[d]_{\pi_A} & B \ar[d]^g \\
     A \ar[r]^f & C
  }
$$
with $\pi_B(a\ot b)=\epsilon(a)b$ and $\pi_A(a\ot b)=\epsilon(b)a$.

The universal property of pullbacks is given in the following way. For $H$  a cocommutative Hopf algebra and
$\gamma: H\to B$ and $\phi: H\to A$ two morphisms, there exists a unique morphism $F: H \to  A{\invamalg}_CB$ 
 making        commutative the diagram

$$
  \xymatrix{
  H\ar@/^/[rrd]^\gamma\ar@/_/[rdd]_\varphi \ar@{.>}[rd]^F \\
    & A{\invamalg}_CB \ar[r]^{\pi_B} \ar[d]_{\pi_A} & B \ar[d]^g \\
    & A \ar[r]^f & C
  }
$$ 
The morphism $F$ is defined by $F(d)= \varphi(d_1)\ot \gamma( d_2)$ for any $d\in H$.

The following theorem is the Theorem $3.7$ in \cite{GKV}.

\begin{thm}
In the category $\Hfc$, the pullback of a cokernel is a cokernel when the ground field has characteristic zero.
\end{thm}

 \begin{rem}
 We do not know if the similar statement for a field of positive characteristic is still true. In fact, the proof of \cite{GKV} uses in an essential way the Cartier-Milnor-Moore theorem which requires the condition on the characteristic of the ground field.
\end{rem}

 \end{section}

  
\section{Coproducts and split epimorphisms}

 The following proposition proves the Axiom A.2   for the category $\Hfc$. 
  
 \begin{prop}
 Let $p$ be  a  morphism in $\Hfc$, $s$  one of its  sections ({\it i.e.} $p{\circ} s= \id_Y$)  in $\Hfc$ and $\kappa:K\to Y$ be the kernel of $p$ in  $\Hfc$ :
  $$\xymatrix{K\ar[r]^{ \kappa } &X\ar[r]_{ p } & Y \ar@/_+2ex/[l]_{s}} $$ 
The arrow $<\kappa,s>:K \coprod Y\to X$ is a cokernel in $\Hfc$.
 \end{prop}
\begin{proof}

First remark that if an element    $x\in X$,   is also an element of $K=\Hker(p)$, then, for any $y\in Y$, on also has $s( y_1) x s(S_Y(y_2))\in K$.  It   can be easily check by straightforward  computation that   the formula  $y\act x=s( y_1) x s(S_Y(y_2))$ defines an action of $Y$ over $K$.  

 Now consider the two linear maps 
  $
  f, g: K\ot Y\to K\coprod Y
  $
  defined by $f( x\ot y) = t_{y_1\act x}t_{y_2}$ and $g(x\ot y)= t_yt_x$
  with $y\in Y$ and $x\in K$.  We denote by $L$ the linear image $\Im(f-g)$.

  Note that both  maps preserve the coalgebra structure  so we have  
  
  $$\begin{aligned}
  \D{\circ} (f -g) = \bigl((f-g)\ot f + g\ot (f-g)\bigr) {\circ} \D.
  \end{aligned}
  $$
  This later relation yields that $L$   is a two-sided coideal of $K\coprod Y$. We set  $U= (K\coprod Y)L (K\coprod Y)$ which is both a two-sided  ideal  a  two-sided coideal.

  Moreover, for any $x\in K$ and $y\in Y$, on computes
  
 \begin{eqnarray*}S(  t_{y_1{\act } x}t_{y_2} - t_yt_x)
  &=&  t_{S(y_1)}   t_{y_2{\act } S(x) }  - t_{S(x)}t_{S(y)}\\
  &=& t_{S(y_1)}    t_{y_2{\act } S(x) }t_{y_3} t_{S(y_3)}   - t_{S(y_1)}t_{y_2}t_{S(x)}t_{S(y_3)}\\
   &=& t_{S(y_1)} \left( ( f-g)(S(x)\ot y_2)\right) t_{S(y_3)}
  \end{eqnarray*}
  Notice that  for the first equality, we used   cocomutativity and the relation $S( y{\act } x)= y{\act }S(x)$ which is a consequence of $S^2=\id$. (Remember that 
    the antipode of a cocommutative Hopf algebra is involutive (cf. \cite[Proposition 4.0.1]{Sw})). 
   Our computation proves $S(U)\subset U$ and consequently that $U$ is a Hopf ideal.

  One clearly has 
 $ 
 (K\coprod Y)/U \simeq  K\# Y$ which is isomorphic to $   X$ after the proof of Lemma~\ref{lm4}.
    Moreover, after Lemma~\ref{lm1},    $ (K\coprod Y)\to  (K\coprod Y)/U$ is a cokernel.    
  \end{proof}  
  
\begin{cor}  If the ground field has characteristic zero, 
  the category $\Hfc$ is finitely cocomplete homological. 
\end{cor}

\begin{proof} The category satisfies to Axioms A1, A2 and A3  (cf \cite{HL}).
\end{proof}

\begin{cor}
 If the ground field has characteristic zero,  the category $\Hfc$ is regular.
\end{cor}

\begin{proof} The result of Proposition 5.1 combined with the existence of finite limits and coequilazers fulfills the definitions axioms of regular categories. 
\end{proof}
    \smallskip

 \begin{section}{Images of kernels}
 It remains to check Axiom A.4. As the category $\Hfc$ is regular, the image of a morphism is  canonically defined     as the coequalizer of its kernel pair (see \cite{Barr}). 
  
Let $f: X\to Y$ be a morphism in $\Hfc$. The coequaliser object of  the kernel pair of $f$ is $ \Hcoeq( \pi_1,\pi_2)$ where $\pi_1$ and $\pi_2$ are the canonical maps of the pullback diagram   $$
 \xymatrix{
    X{\invamalg_Y} X \ar[r]^{ \pi_1}\ar[d]_{ \pi_2 } & X \ar[d]^{ f}\\ 
     X  \ar[r]_{f }  & Y} 
 $$
  We have $  X{\invamalg_Y} X = \{x\ot x'\in X\ot X\mid x_1\ot f(x_2)\ot x'= x \ot f(x'_1)\ot x'_2 \}$ and $\pi_1(x\ot x')=  \e(x')x$ and  
  $\pi_2(x\ot x')=  \e(x )x'$.
  
One has   $ \Hcoeq( \pi_1,\pi_2)=  X  /  XJX$  where  $J=\{  \e(x')x- \e(x )x' \mid  x\ot x'\in X\ot_Y X    \}$.
  It is easy to see that $J$ is in fact a Hopf ideal of $X$ so $ XJX= J$.  The image object $   X  /   J $ of $f$ will be  denoted   by   $\HIm(f) $. 
  
 After \cite{BB} the morphism $f$ factorizes  as a product of the regular epimorphism $\pi=\hcoeq( \pi_1,\pi_2)$ and a monomorphism $\iota$. One has the diagram 
 $$
 \xymatrix{
    X\ar[r]^{ f }\ar[d]_{ \pi } & Y \\ 
  \HIm(f) \ar@/_0.2pc/[ru]_{ \iota } } 
 $$
  the morphism $\iota$ being  induced by $f: X\to Y$.
   In our case, in fact, the notion of image and linear image coincide.

 \begin{lm}  If the ground field has characteristic zero,  
 in the category $\Hfc$, for any  morphism  $f$, one has $\HIm(f)\cong \Im(f)$. 
 \end{lm}
 
\begin{proof}
 The following factorization diagram derives from the construction of  $\iota$:
 $$
 \xymatrix{
    X\ar[r]^{ f }\ar[d]_{ \pi } & \Im(f)   \ar[r]^{ \  \subset }  &Y\\ 
  \HIm(f) \ar@/_0.2pc/[ru]_{ \iota' } \ar@/_0.8pc/[rru]_{ \iota }  } 
 $$
where $\iota'$ is  still a monomorphism  but is also surjective.  After Lemma \ref{lm3} a monomorphism is injective.    \end{proof}

We can now      prove  that Axiom  A.4 is fullfilled.
  \begin{prop}  If the ground field has characteristic zero,  
    in $\Hfc$ the image of a  kernel is a  kernel.
 \end{prop}

\begin{proof} Consider the following commutative diagram in $\Hfc$:
     $$
     \xymatrix{
     &  A\ar[d]^{f}\\
    \Hker(g)\ar[r]^{\ \ \hker(g)}\ar[rd]^{\pi' }\ar[d]^{\pi} &  X \ar[r]^{ g }\ar[d]^{\hcoker(f) } &Z\\ 
 \HIm(\pi')\ar[r]^{\iota \ } &  \Hcoker(f) \\ } 
     $$

     As $ \Hker(g)$ is a normal sub-Hopf algebra of $X$, its linear image,   through the projection  $X\to \Hcoker(f)$ is   a  
      normal sub-Hopf algebra of $\Hcoker(f)$. The later linear image is nothing else than $\HIm(\pi')$. After Lemma \ref{lm2}, it is a kernel object under our assumptions. 
      \end{proof}
     
     At this point, we proved that all axioms A1, A2, A3 and A4 are fulfilled for $\Hfc$ and so one recovers  Theorem 0.1.

 \end{section}

 
 \begin{section}{The abelian core, the categorical semi-abelian product}
 This section is widely inspired by \cite{Bo}. 
In a first time, we determine the abelian core of $\Hfc$. In a second time, we prove that the categorical semi-direct product in $\Hfc$ is nothing else than the semi-direct product defined in section 3.    In all this section, we assume that the ground field has characteristic zero.

\begin{lm}
Let $A$ be a sub-algebra of a cocommutative Hopf algebra $H$. The sub-Hopf algebra is normal if and only if the inclusion $A\to H$ is normal in $\Hfc$.
\end{lm}

\begin{proof}   Consider a normal  map $A\to H$  in $\Hfc$. The sub-object  $A$ is a sub-Hopf algebra of $H$ such as it exists a morphism $\phi: H\to H'$ and $A=\Hker(\phi)$. We already noticed that kernel objects are normal sub-Hopf algebras. 
The converse assertion is     Lemma \ref{lm2}.
\end{proof}

 The following proposition is Theorem 0.2. 
 \begin{prop}
 The full sub-category  of abelian objects of $\Hfc$ is $\Hfcc$. \end{prop}
 
\begin{proof}
 We use the characterization of   \cite[Theorem 6.9]{Bo} which states that an object $C$ in a semi-abelian category is abelian if and only if its diagonal $C\to C\ot C$ is normal. In our case, if $C$ is an object of $\Hfc$, the diagonal map is nothing else than the comutiplication $\Delta_C$.  So after Lemma 7.1, it suffices to prove that $C$ is abelian if and only if 
 $\Im(\Delta_C)$ is a normal sub-Hopf algebra of $C\ot C$.
 
 If $C$ is commutative,   so is $C\ot C$ and as  sub-Hopf algebras of a commutative algebra are normal, it follows that $\Im(\Delta_C)$ is. 
 
 On the other hand suppose that $\Im(\Delta_C)$ is a normal sub-Hopf algebra of $C\ot C$. For any two elements $a,c\in C$ it exists $d\in C$ such that we have
 $$
 \Delta(d)= \bigl(c_1\ot 1 \bigr)\bigl(a_1\ot a_2 \bigr)S\bigl(c_2\ot1 \bigr)= c_1a_1S(c_2)\ot a_2.
 $$
 Successively applying $\e_C$ to each tensor factor  of the previous equality, we get 
 $$
d= \e_C(c)a= c_1aS(c_2).
 $$
 As the identity is true for any $a\ot c\in C$, we may apply it to the first tensor factor of $ a\ot c_1\ot c_2$
and get 
$$
\e_C(c_1)a\ot c_2= c_1aS(c_2)\ot c_3 \Longrightarrow  \e_C(c_1)a c_2= c_1aS(c_2)  c_3 \Longrightarrow ac=ca.
$$ Thus, $C$ is commutative. 
\end{proof}

 We retrieve known results: 
 the category $\Hfcc$ is abelian    and the category  $\Hfc$ is not abelian.

To end the article,   we prove that the semi-direct product defined in Section 3 is the semi-direct product in $\Hfc$ in the categorical sense defined in  \cite{Bo}. We will follow the latter reference. 
 
 For any object $Y$ in $\Hfc$ we have a pair of adjoint functors: 
 
 \centerline{
 \begin{minipage}{150pt}
 \begin{eqnarray*}
 \Hker: &  \Pt_Y  \longrightarrow   \Hfc\\
   &\xymatrix{  X\ar[r]_{ p } & Y \ar@/_0.7pc/[l]_{s}}   \mapsto  \Hker(p)\\
  \end{eqnarray*}   \end{minipage}
  \qquad and \qquad
   \begin{minipage}{150pt} 
   \begin{eqnarray*}
  &  \Hfc  \longrightarrow      \Pt_Y \\
   &K   \mapsto     \xymatrix@!0  @C=80pt{   K\coprod Y \ar[r]_{\  (\e_K, \scriptstyle {\rm id}_y) }  & Y \ar@/_0.8pc/[l]_{\ \ \iota_Y}}    \\
  \end{eqnarray*}  
  \end{minipage}}
  
\noindent  where the functor $\Hker$ is monadic. Then one can consider the monad $\mathbb T_Y$ associated to $\Hker$. By definition, the semi-direct product of an algebra $(K,\xi)$  for the monad $\mathbb T_Y$ and the object $Y$ is the domain of the pointed object $ (p,s): \xymatrix{  X\ar[r]_{ p } & Y \ar@/_0.7pc/[l]_{s}} $ corresponding to $(K,\xi)$ via the equivalence $\Pt_Y\cong (\Hfc)^{\mathbb T_Y}$.

  \begin{thm}
Let   $Y$ be an object in $\Hfc$. Let $(K,\xi)$  be an algebra fo the monad $\mathbb T_Y$. The~semi-direct product of an algebra $(K,\xi)$ and     $Y$ is  $K\#Y$.
  \end{thm}
  
  \begin{proof}
  The proof given for the category of groups in \cite[Section 5]{Bo} is still valid in our case if one replaces \cite[Proposition 5.7]{Bo} by Lemma \ref{lm4}.
  \end{proof}

 \end{section}

 \goodbreak

  \bibliographystyle{alpha}
\bibliography{cathopf-bib}

    \end{document}